\newcommand{\R}{\mathds R}

\documentclass[twoside,a4paper,10pt, draft]{amsart}

\usepackage{times}
\usepackage{dsfont}
\numberwithin{equation}{section}

\title[Unparameterized embeddings]{On the manifold structure of the set of unparameterized embeddings with low regularity}
\author[L. J. Al\'\i as]{ Luis J. Al\'\i as}
\address{Departamento de Matem\'{a}ticas, \hfill\break\indent
Universidad de Murcia, Campus de Espinardo\hfill\break\indent 30100 Espinardo,
Murcia, \hfill\break\indent Spain} \email{ljalias@um.es}

\author[P.\ Piccione]{Paolo Piccione}
\address{Departamento de Matemática,\hfill\break\indent
Universidade de S\~ao Paulo, \hfill\break\indent Rua do Mat\~ao
1010,\hfill\break\indent CEP 05508-900, S\~ao Paulo, SP, Brazil}
\email{piccione.p@gmail.com}
\thanks{This work was partially supported by MEC project MTM2009-10418, and Fundaci\'{o}n S\'{e}neca
project 04540/GERM/06, Spain. This research is a result of the activity developed
within the framework of the Programme in Support of Excellence Groups of the
Regi\'{o}n de Murcia, Spain, by Fundaci\'{o}n S\'{e}neca, Regional Agency for
Science and Technology (Regional Plan for Science and Technology 2007-2010).
The research leading to this manuscript was carried out while the second author was
visiting the \emph{Departamento de Matem\'aticas} of the \emph{Universidad de Murcia}, Spain, as a
Visiting Professor. P.P. would like to thank faculties and staff of this institution for the
warm hospitality and for the excellent working conditions provided to him.}
\date{July 25th, 2010}

\begin{document}

% Theorems and such

\theoremstyle{plain}\newtheorem*{teon}{Theorem}
\theoremstyle{definition}\newtheorem*{defin*}{Definition}
\theoremstyle{plain}\newtheorem{teo}{Theorem}[section]
\theoremstyle{plain}\newtheorem{prop}[teo]{Proposition}
\theoremstyle{plain}\newtheorem{lem}[teo]{Lemma}
\theoremstyle{plain}\newtheorem{cor}[teo]{Corollary}
\theoremstyle{definition}\newtheorem{defin}[teo]{Definition}
\theoremstyle{remark}\newtheorem{rem}[teo]{Remark}
\theoremstyle{plain} \newtheorem{assum}[teo]{Assumption}
\swapnumbers
\theoremstyle{definition}\newtheorem{example}{Example}[section]
\theoremstyle{plain} \newtheorem*{acknowledgement}{Acknowledgements}
\theoremstyle{definition}\newtheorem*{notation}{Notation}

%%%%%

\begin{abstract}
Given manifolds $M$ and $N$, with $M$ compact, we study the geometrical structure
of the space of embeddings of $M$ into $N$, having less regularity than $\mathcal C^\infty$,
quotiented by the group of diffeomorphisms of $M$.
\end{abstract}

\maketitle

\begin{section}{Introduction}
A very general class of geometrical variational problems can be formulated in terms of
some action functional defined on the space $\mathrm{Emb}(M,N)$
of embeddings of a manifold $M$ into some other manifold $N$. In many interesting examples,
as for instance in the study of \emph{minimal} or \emph{constant mean curvature} embeddings
$x:M\to N$, the functionals involved do not depend on the parameterization $x$, i.e., they are
invariant by $\mathrm{Diff}(M)$ the diffeomorphism group of $M$ that acts by right composition
on the space of embeddings. Under these circumstances, given a solution $x:M\to N$ of the
variational problem, any embedding of the form $x\circ\phi$, with $\phi\in\mathrm{Diff}(M)$,
is also a solution of the problem, which is not geometrically distinct from $x$.
This implies in particular, that typical compactness assumptions, like the Palais--Smale condition,
obviously fail for parameterization invariant functionals. Namely, every critical level
of a parameterization invariant functional is non compact.
If one is interested in multiplicity results, like for instance Morse Theory or Bifurcation Theory, one
has to identify solutions that are not geometrically different.
There are several methods in the literature to get rid of the gauge invariance property in equivariant variational
problems. One method is to impose a \emph{gauge fixing condition}, in the language of \cite{PalaisTerng},
i.e., a smooth submanifold of the domain of the functional, which intersects all the orbits of the group action,
and on which the variational problem has no invariance properties.
A second method consists in determining an auxiliary functional, with \emph{the same} critical points
and which is no longer gauge invariant. This is illustrated well in the classical closed geodesic problem,
originally formulated using the length functional
in the space of immersions of the circle in a Riemannian manifold $N$. In this case, one
replaces the length functional by a quadratic energy functional, which is no longer parameterization
invariant, and has the same critical points.
Nonetheless, the same technique may not be available for variational problems in higher dimension, and
in this case the appropriate functional space to consider for the variational problem is the set of
\emph{unparameterized embeddings} of $M$ into $N$.
Two embeddings $x_1,x_2:MN\to N$ are said to be equivalent if there exists a diffeomorphism $\phi$ of $M$
such that $x_2=x_1\circ\phi$; an unparameterized embedding of $M$ into $N$ is an equivalence
class of embedding of $M$ into $N$.
Actions of the diffeomorphism group of a manifold have been studied in several contexts, and one
of the central questions is how to construct \emph{slices} for these actions.
The interested reader may look up \cite{Ebin} for the action on Riemannian metrics by pull-back,
or \cite{FreUhl} for gauge theory.

A crucial point is the choice of regularity for the embeddings. Namely, important properties
of the variational problem, like for instance the Palais--Smale condition, or the Fredholmness
condition for the second derivative, depend essentially on this choice.
The $\mathcal C^\infty$ case has been extensively studied (see \cite{Michor, Michor1, Michor2}),
and a nice Frechet differentiable structure has been described
for this set. The theory of manifolds modeled on general locally convex topological
vector spaces has been recently developed in detail in \cite{KriMich}.
Nevertheless, in view to applications in variational calculus, the Frechet
structure of $\mathcal C^\infty$ embeddings is too weak, and it is desirable to have
a geometry modeled on Banach or Hilbert spaces.
Usually, a natural choice would be to consider embeddings
of class $\mathcal C^k$, or $\mathcal C^{k,\alpha}$, with $k<\infty$ and $\alpha\in\left]0,1\right[$,
or some Sobolev regularity.
However, when a regularity weaker than $\mathcal C^\infty$ is assumed for the embeddings,
subtle obstructions arise when attempting to define a global differentiable
structure on the quotient space of embeddings modulo diffeomorphisms.
The problem is a consequence of the fact that, when $k<\infty$, the map of left-composition
with a fixed diffeomorphism of class $\mathcal C^k$ is \emph{not} a differentiable
map in the space of $\mathcal C^k$-maps. The transition maps of any natural atlas of charts
for the space of unparameterized embeddings involve this type of operations.

The point we address in this paper is precisely an analysis of the local and global
geometrical structure of the set of unparameterized embeddings having regularity
weaker than $\mathcal C^\infty$. We will show that, unlike the smooth case, such a set does not
have a natural global differentiable structure; nonetheless local and global techniques
from the Calculus of Variations can be applied for parameterization invariant functionals.
More precisely, we use Palais' notion of Vector Bundle Neighborhood (VBN) for describing
an atlas of charts for the set of unparameterized embeddings, whose transition functions
are continuous. Using these charts, the set of unparameterized embeddings is ``locally'' a smooth
submanifold of the space of embeddings.  The restriction of any parameterization
invariant smooth function on the space of embeddings defines a function on the
space of unparameterized embeddings which is smooth in any local chart.
Thus, one has a well defined notion of critical point, and we compute the first
and the second variation at a critical point of a parameterization invariant
smooth functional.
In the last section we also analyze regularity properties of the action of the isometry group of
the target manifold $N$ on the space of unparameterized embeddinds by left-composition.
This action is also not smooth, but in local charts its orbits are smooth embedded submanifolds.

\end{section}
\begin{section}{Notations and preliminaries}
Let us consider two  smooth (i.e., $\mathcal C^\infty$) manifolds $M^m$ and $N^n$, with $m<n$.
For simplicity, we will assume that $M$ is compact, although an analogous theory can be developed
also in the non compact case, along the lines of \cite{PicTau}.
We will fix throughout an auxiliary Riemannian metric $g$ on the target manifold $N$, and we will
denote by $\exp$ the corresponding exponential map. The metric $g$ induces a norm on every vector bundle
obtained by functorial construction from $TN$ (like pull-backs, normal bundles of embeddings into $N$, etc.).
The metric $g$ will be used only for a more explicit description of the manifold charts;
all the results of the present paper will not depend on the choice of such metric.

We will denote by $\mathfrak C$ a regularity class of maps defined on $M$. More precisely,
let $\mathfrak C(M,\mathds R)$ be a Banach space of maps from $M$ to $\mathds R$ such that
\[\mathcal C^\infty(M,\R)\subset\mathfrak C(M,\mathds R)\subset \mathcal C^1(M,\mathds R),\]  with \emph{dense} inclusion
$\mathcal C^\infty(M,\R)\hookrightarrow\mathfrak C(M,\mathds R)$ and \emph{continuous} inclusion
$\mathfrak C(M,\mathds R)\hookrightarrow \mathcal C^1(M,\mathds R)$.
We require that $\mathfrak C(M,\mathds R)$ be
stable under composition from the right with functions $f\in\mathcal C^\infty(M,M)$ (this action is linear),
and stable under composition from the left with functions $f\in\mathcal C^\infty(\mathds R,\mathds R)$.
We also assume that for all $f\in\mathcal C^\infty(\mathds R,\mathds R)$, the map $\mathfrak C(M,\mathds R)\ni g\mapsto g\circ f\in\mathfrak C(M,\mathds R)$
is smooth.

Typical examples of $\mathfrak C$ are:
\begin{itemize}
\item $\mathfrak C=\mathcal C^k$, with $k\ge1$;
\item $\mathfrak C=\mathcal C^{k,\alpha}$, with $k\ge1$ and $\alpha\in\left]0,1\right[$ (H\"older type regularity);
\item $\mathfrak C=W^{k,p}$, with $p(k-1)>m$ (Sobolev type regularity).
\end{itemize}
In several interesting examples, also non standard choices for the functor $\mathfrak C$ may
appear naturally, see Remark~\ref{thm:remseparabilita}. Thus, treating the subject in such generality is
not a useless abstraction.

A description of the differentiable structure of the set of maps $f:M\to N$ of class
$\mathfrak C$ can be given as follows.
Set $\mathfrak C(M,\R^d)=\oplus_{i=1}^d\mathfrak C(M,\R)$ and, given a subset $S\subset\R^d$,
denote by $\mathfrak C(M,S)$ the set of maps $f\in\mathfrak C(M,\R^d)$ such that $f(M)\subset S$.
Such set is endowed with the induced topology from $\mathfrak C(M,\R^d)$.
If $S$ is a submanifold of $\mathds R^d$, then $\mathfrak C(M,S)$ is a submanifold of $\mathfrak C(M,\mathds R^d)$.
Given a smooth embedding $\phi:N\to\R^d$, denote by $\mathfrak C(M,N,\phi)$ the set of all
maps $f:M\to N$ such that $\phi\circ f\in\mathfrak C\big(M,\phi(N)\big)$. The
map $\mathfrak C(M,N,\phi)\ni f\mapsto\phi\circ f\in\mathfrak C\big(M,\phi(N)\big)$ is a bijection,
and it induces a a Banach manifold structure on $\mathfrak C(M,N,\phi)$.
This differentiable structure is independent on $\phi$, i.e.,
given different embeddings $\phi_i:N\to\mathds R^{d_i}$, $i=1,2$, then $\mathfrak C(M,N,\phi_1)=\mathfrak C(M,N,\phi_2)$, and the differentiable structures induced by $\mathfrak C\big(M,\phi_1(N)\big)$
and $\mathfrak C\big(M,\phi_2(N)\big)$ coincide. We will therefore omit the symbol $\phi$ in the
notation of the set of maps $f:M\to N$ of class $\mathfrak C$, and we will write $\mathfrak C(M,N)$.
Given a smooth vector bundle $\pi:E\to M$, one also has a notion of \emph{sections of class $\mathfrak C$}
of $\pi$, defined in the obvious way.
\begin{rem}\label{thm:remseparabilita}
When the Banach space $\mathfrak C(M,\R)$ is not separable, as in the case $\mathfrak C=\mathcal C^{k,\alpha}$, with
$\alpha\in\left]0,1\right[$, then $\mathfrak C(M,N)$ is a non separable Banach manifolds.
There are theories where separability is an important issue, especially when Sard's theorem
needs to be invoked.
A situation of this type is considered in \cite{Whi2}, where the author proves a genericity result
in the space of $\mathcal C^{k,\alpha}$ embeddings. As suggested in \cite[\S~1.5]{Whi}, a possible way of circumventing the problem is to consider rather than the space $\mathcal C^{k,\alpha}$,
the closed subspace $\mathcal C^{k,\alpha+}$consisting of all $\mathcal C^{k,\alpha}$-limits of functions of class $\mathcal C^{k+1}$.
This space is separable with respect to the $\mathcal C^{k,\alpha}$-topology, and in fact it is second countable.
\end{rem}
By the assumption that the inclusion $\mathfrak C(M,\mathds R)\hookrightarrow \mathcal C^1(M,\mathds R)$ is continuous,
it follows that the (possibly empty) subset of $\mathfrak C(M,N)$ consisting of embeddings is open.
In next Section we will describe an explicit set of local charts for such set, given intrinsically, i.e.,
without using embeddings of $N$ into some Euclidean space.
\end{section}
\begin{section}{The manifold of embeddings}
\label{sec:diffstructure}
Classical references where the differentiable structure
of $\mathfrak C(M,N)$, or more generally of spaces of $\mathfrak C$-sections of fiber bundles\footnote{%
functions from $M$ to $N$ can be thought of as sections of the trivial fiber bundle $M\times N$.}
with compact base, has been described explicitly are \cite{Eliasson, Eells, Palais};  local charts of this
differentiable structure are described by Palais using the notion
of \emph{vector bundle neighborhood} (VBN). When the base is non compact, restrictions on the space
of sections are required in order to have a well defined Banach differentiable structure, see \cite{PicTau}.
In order to get a better insight on our problem,
let us recall how a global differentiable structure on $\mathfrak C(M,N)$ is
obtained, following the VBN approach of \cite{Palais}.
Given a Riemannian vector bundle $E$ over $M$ (i.e., a vector bundle endowed with a Riemannian
structure on the fibers and a compatible connection), we will denote by $\mathbf\Gamma(E)$ the
Banach space of all sections of class $\mathfrak C$ of $E$.
The essential property required for developing Palais' theory is the fact, proved in \cite{Palais}, that,
given a compact manifold $M$, two Riemannian vector bundles $E_1,E_2$ over $M$, and a smooth
vector bundle morphism $\Phi:E_1\to E_2$, the composition operator $\mathbf\Gamma(E_1)\ni s\mapsto\Phi\circ s\in
\mathbf\Gamma(E_2)$ is a smooth map.
The idea of vector bundle neighborhoods is that
suitable small $\mathfrak C$-neighborhoods of a given map $x:M\to N$ of class $\mathcal C^\infty$ are parameterized by elements in
neighborhoods of the zero section of the pull-back bundle $x^*(TN)$ over $M$. More precisely,
once a Riemannian metric $g$ with Levi--Civita connection $\nabla$ in $N$ is fixed,
a local chart $\Phi$ of $\mathfrak C(M,N)$ around a given smooth function $x$ is obtained by
associating to each section $u$ of class $\mathfrak C$ of the vector bundle $x^*(TN)$ the map
$y:M\to N$ defined by $y(p)=\exp_{x(p)}\big(u(p)\big)$, where $\exp$ is the exponential
map of $\nabla$. The inverse of the map that associates to each $u$ the corresponding $y$ defines
a local chart from an open neighborhood of the zero section of $x^*(TN)$ to an open neighborhood
of $x$, that will be denoted by $\Phi_x$.
The transition maps for charts in this atlas are computed as follows.
Given smooth maps $x_1,x_2:M\to N$, for $i=1,2$ consider the map
$\mathrm{EXP}_i:x_i^*(TN)\to M\times N$  defined by $\mathrm{EXP}_i(p,v)=\big(p,\exp_{x_i}(v)\big)$, $v\in T_{x_i(p)}N$.
This gives a smooth diffeomorphism of an open subset containing the zero section of $x_i^*(TN)$ onto
an open neighborhood of the graph of $x_i$; the composition $\zeta=\mathrm{EXP}_2^{-1}\circ\mathrm{EXP}_1$
is a smooth diffeomorphisms between two open neighborhood of the zero sections of the vector bundles
$x_1^*(TN)$ and $x_2^*(TN)$ \emph{that preserves the fibers}. The transition map $\Phi_{x_1}^{-1}\circ\Phi_{x_2}$
is given by left-composition with the smooth map $\zeta$, and thus it is differentiable
(compare with the situation described in Remark~\ref{thm:remlossofdiff}).
Moreover, when $x$ varies in the set of smooth functions,
the domain of these charts cover the entire $\mathfrak C(M,N)$, as we are assuming density of the
inclusion $\mathcal C^\infty(M,\R)\hookrightarrow\mathfrak C(M,\R)$.
Hence, the collection of all such charts defines a differentiable atlas on $\mathfrak C(M,N)$.
Given a smooth map $x:M\to N$, the tangent space $T_x\mathfrak C(M,N)$ is identified, via the
chart $\Phi_{x}$, with the space of all sections of class $\mathfrak C$ of the pull-back bundle $x^*(TN)$.

The subset $\mathrm{Emb}(M,N)$ of $\mathfrak C(M,N)$ consisting of all embeddings
$x:M\to N$ is open, and thus it inherits a natural Banach manifold structure from
$\mathfrak C(M,N)$. One can consider the set $\mathrm{Diff}(M)$,
which is  the set of all diffeomorphisms $\phi:M\to M$ of class $\mathfrak C$;
observe that $\mathrm{Diff}(M)$ may fail to be closed under composition or inverse, so
that in general it is not a group.
$\mathrm{Diff}(M)$ is an open subset of $\mathrm{Emb}(M,M)$, and thus
it inherits a natural differentiable structure.
However, even under the assumption that $\mathrm{Diff}(M)$ is closed under
composition and inverse, \emph{neither one of the two operations is differentiable}.
Namely, the left-composition map $\phi\mapsto x\circ\phi$ on $\mathrm{Diff}(M)$ in general is not
of class $\mathcal C^1$ (see \cite[Appendix]{Whi}).
Similarly, the derivative of the map $\phi\mapsto\phi^{-1}$ involves the derivative of $\phi$,
and thus this is not differentiable at those points $\phi$ whose derivative is not of class $\mathfrak C$.
\end{section}
\begin{section}{The manifold of unparameterized embeddings}\label{sub:tilde}
Two embeddings $x,y:M\to N$ will be considered equivalent if there exists a $\mathcal C^1$-diffeomorphism
$\phi:M\to M$ such that $y=x\circ\phi$, i.e., if they are different parameterizations of the same
submanifold of $N$ diffeomorphic to $M$. If $x$ and $y$ are of class $\mathcal C^k$, then
such diffeomorphism $\phi$ will also be of class $\mathcal C^k$.
For $x\in\mathrm{Emb}(M,N)$, we will denote by $[x]$ the class of all
$y\in\mathrm{Emb}(M,N)$ that are equivalent to $x$.
\begin{defin}
The set of \emph{unparameterized embeddings of class $\mathfrak C$ of $M$ into $N$}, denoted by
$\widetilde{\mathrm{Emb}}(M,N)$, is the set:
\[\widetilde{\mathrm{Emb}}(M,N)=\big\{[x]:x\in\mathrm{Emb}(M,N)\big\}.\]
\end{defin}
Thus, $\widetilde{\mathrm{Emb}}(M,N)$ can be thought as the set of all embedded submanifolds
of class $\mathfrak C$ of $N$ that are $\mathfrak C$-diffeomorphic to $M$. We will now establish an infinite dimensional
Banach \emph{topological} structure on $\widetilde{\mathrm{Emb}}(M,N)$, and we will describe
suitable local charts of this structure.

Let $x:M\to N$ be a smooth embedding; a local chart $\widetilde\Phi:\widetilde{\mathcal U}_{x}\to
\widetilde{\mathcal W}_{x}$
in $\widetilde{\mathrm{Emb}}(M,N)$, where $\widetilde{\mathcal U}_{x}$ is an appropriate neighborhood
of $[x]$ in $\widetilde{\mathrm{Emb}}(M,N)$, ${\mathcal W}_{x}$ is an appropriate $\mathfrak C$-neighborhood
of the zero section of the normal bundle of $x$, is given
as follows. There exists an open subset $U$ of the normal bundle $x^\perp$ containing the zero section
of this bundle, and an open subset $V$ of $N$ containing the image $x(M)$ such that
the restriction of $\exp$ to $U$ gives a diffeomorphism from $U$ to $V$.
The space $\mathbf\Gamma(x^\perp)$ of all sections of class $\mathfrak C$ of the normal bundle $x^\perp$ is a
Banach space, and the subset $\mathbf\Gamma(x^\perp;U)$ of $\mathbf\Gamma(x^\perp)$ consisting
of all sections whose image is contained in $U$ is open. A map
$\widetilde\Psi_{x}:\mathbf\Gamma(x^\perp;U)\to\widetilde{\mathrm{Emb}}(M,N)$ is obtained
by setting $\widetilde\Psi_{x}(u)=[y]$, where $y(p)=\exp_{x(p)}\big(u(p)\big)$ for all $p\in M$.
Clearly, $y$ is an embedding of class $\mathfrak C$ of $M$ into $N$, since $u$ is an embedding of class $\mathfrak C$
of $M$ into the normal bundle $x^\perp$, and $\exp$ is a diffeomorphism from $U$ to $V$.
It is easy to see that $\widetilde\Psi_{x}$ is injective. In order to prove this, first observe that
two embeddings $x_1,x_2\in\mathrm{Emb}(M,N)$ are equivalent
if and only if $x_1(M)=x_2(M)$. Now, observe that two distinct sections $u_1,u_2\in\mathbf\Gamma(x^\perp;U)$
must have distinct images in $U$, and thus their composition with $\exp$ are also different in $V$.
This proves that $\widetilde\Psi_{x}$ is injective. The image of $\widetilde\Psi_{x}$ is
the projection onto $\widetilde{\mathrm{Emb}}(M,N)$ of an open neighborhood of $x$ in ${\mathrm{Emb}}(M,N)$.
If $y\in{\mathrm{Emb}}(M,N)$ is near $x$, in particular it has image contained in $U$, then
$\exp^{-1}\big(y(M)\big)$ is the image of a section $u$ of $x^\perp$ of class $\mathfrak C$; then,
$\widetilde\Psi_{x}(u)=[y]$.
Thus, the map $\Psi$ is a bijection from an open subset ${\mathcal W}_{x}$ of $\mathbf\Gamma(x^\perp)$
containing the zero section, to a subset ${\mathcal U}_{x}$ of $\widetilde{\mathrm{Emb}}(M,N)$
given by the projection onto $\widetilde{\mathrm{Emb}}(M,N)$
of an open neighborhood of $x$ in ${\mathrm{Emb}}(M,N)$. Its inverse will be denoted by $\widetilde\Phi_{x}$,
and the collection of such maps, as $x$ varies in the set of all smooth embeddings of $M$ into $N$
is taken as an atlas of charts for $\widetilde{\mathrm{Emb}}(M,N)$.

We note however that \emph{there is no differentiable compatibility between two charts in this atlas}, i.e.,
the transition maps are in general \emph{not} differentiable, but only continuous.
Let us compute a transition map. Denote by $x_1,x_2:M\to N$ two smooth embeddings such that
the classes $[x_1]$ and $[x_2]$ belong to the intersection of the domains $\widetilde{\mathcal U}_{x_1}\cap
\widetilde{\mathcal U}_{x_2}$
of the charts  $\widetilde\Phi_{x_1}$ and $\widetilde\Phi_{x_2}$.
Denote by $\exp_1$, $\exp_2$ the exponential map of $g$ restricted
to the normal bundles $x_1^\perp$ and $x_2^\perp$ respectively, that are diffeomorphisms
between open subsets containing the zero section and tubular neighborhoods of the images $x_1(M)$
and $x_2(M)$ respectively.
Thus, there are open subsets $U_i\subset x_i^\perp$ containing the zero section
such that the map $\zeta:U_1\to U_2$ given by $\zeta=\exp_2^{-1}\circ\exp_1$ is a smooth diffeomorphism.
Let $u\in\widetilde{\mathcal W}_{x_1}\cap\widetilde{\mathcal W}_{x_2}$ be fixed and set
$u'=\widetilde\Phi_{x_2}\big(\widetilde\Phi_{x_1}^{-1}(u)\big)$.
\begin{rem}\label{thm:remlossofdiff}
The key observation here is that, in spite of the fact that the section $u'$ of the normal bundle $x_2^\perp$
has the same image of the map $\zeta\circ u$, the latter is \emph{not} a section of $x_2^\perp$.
This depends on the fact that the diffeomorphism $\zeta$ is not a vector bundle morphism
as in the case of the charts of $\mathrm{Emb}(M,N)$ (Section~\ref{sec:diffstructure}), i.e., it does
not take fibers of $x_1^\perp$ into fibres of $x_2^\perp$. In order to obtain the section $u'$, an \emph{adjustment}
needs to be done in the domain of $\zeta\circ u$, which is obtained by composition on the right with a diffeomorphism
of the base $M$ that depends on $u$; it is precisely such adjustment that causes the loss of differentiability of the
transition maps.
\end{rem}
The following formula holds:
\[u'=\zeta\circ u\circ h_u^{-1},\]
where $h_u:M\to M$ is the diffeomorphism:
\[h_u=\pi_2\circ\zeta\circ u,\]
$\pi_2:E_2\to M$ being the projection of the vector bundle $E_2$ over the base manifold $M$.
Now, the maps $u\mapsto\zeta\circ u$ and $u\mapsto h_u$ are $\mathcal C^\infty$, but the function $h\to h^{-1}$ is
not differentiable in $\mathrm{Diff}(M)$ where $h$ is only of class $\mathfrak C$,
as well as the function of composing on
the left with $\zeta\circ u$, when $u$ is only of class $\mathfrak C$. Thus, the map $u\mapsto u'$ is
continuous, but not differentiable.

We can then define a unique topology on $\widetilde{\mathrm{Emb}}(M,N)$ whose basis is the collection
of the domains $\widetilde{\mathcal U}_{x}$ of the charts $\widetilde\Phi_{x}$, as $x$ varies in the set
of smooth embeddings of $M$ into $N$, and by requiring that each $\widetilde\Phi_{x}$ is a homeomorphism
onto its image. It is easy to see\footnote{Consider the restriction of $\widetilde\pi$ to the inverse
image $\widetilde\pi^{-1}\big(\widetilde{\mathcal U}_{x}\big)$ of the domain of some chart.
Then, such restriction is continuous, open (because it admits continuous local sections with arbitrarily prescribed
values at a given point), and surjective, hence it
is a quotient map.} that this topology is exactly the quotient topology induced by
the canonical quotient map $\widetilde\pi:{\mathrm{Emb}}(M,N)\to\widetilde{\mathrm{Emb}}(M,N)$.

The reader should observe that the charts $\Phi_{x}$ in $\mathrm{Emb}(M,N)$ and
$\widetilde\Phi_{x}$ in $\widetilde{\mathrm{Emb}}(M,N)$ look very much alike.
The only difference is that $\widetilde\Phi_{x}$ takes values in the space of sections
of the normal bundle $x^\perp$, while $\Phi_{x}$ takes values in the spaces of sections of $x^*(TM)$.
If we identify\footnote{%
We will identify the pull-back bundle $x^*(TN)$ with the Whitney sum $\mathrm dx(TM)\oplus x^\perp$.}
$x^\perp$ with a subbundle of $x^*(TN)$, then this suggests that, roughly speaking,
``locally $\widetilde{\mathrm{Emb}}(M,N)$ is a smooth submanifold of $\mathrm{Emb}(M,N)$''.
Let us state this in a more precise way:
\begin{prop}
For $x$ varying in the set of smooth embeddings of $M$ into $N$, the family
$\Big\{\big(\widetilde{\mathcal U}_{x},\widetilde\Phi_{x}\big)\Big\}_x$
is an atlas of charts of $\widetilde{\mathrm{Emb}}(M,N)$, whose domains form an open cover of
$\widetilde{\mathrm{Emb}}(M,N)$,
and that makes $\widetilde{\mathrm{Emb}}(M,N)$ into an infinite dimensional \emph{topological} manifold
modeled on the Banach space $\mathfrak C(M,\R^{n-m})$.

The canonical projection $\widetilde\pi:{\mathrm{Emb}}(M,N)\to\widetilde{\mathrm{Emb}}(M,N)$
is a quotient map.

For a given smooth embedding $x:M\to N$, by identifying
the normal bundle $x^\perp$ with a subbundle of
the pull-back $x^*(TN)$, then the local chart $\Phi_{x}$ of $\mathrm{Emb}(M,N)$ around
$x$ and the local chart $\widetilde{\mathrm{Emb}}(M,N)$ around $[x]$ allow an identification
of the neighborhood $\widetilde U(x)$ of $[x]$ with the smooth submanifold of $\mathrm{Emb}(M,N)$
consisting of those $\mathfrak C$-embeddings in the domain of the chart $\Phi_{x}$ for which $\Phi_{x}$
takes values in the space of sections of the normal bundle $x^\perp$.\qed
\end{prop}
The local identification of $\widetilde{\mathrm{Emb}}(M,N)$ with submanifolds of ${\mathrm{Emb}}(M,N)$
is particularly useful for studying smooth maps.
\begin{cor}\label{thm:corregularityinvfunct}
Let $\mathfrak Z$ be an arbitrary manifold and $f:\mathrm{Emb}(M,N)\to\mathfrak Z$
be a smooth function such that $f(x)=f(y)$ for all
pairs of equivalent embeddings $x,y\in\mathrm{Emb}(M,N)$. Then, given any smooth embedding $x:M\hookrightarrow N$,
considering the local chart $\big(\widetilde{\mathcal U}(x),\widetilde\Phi_{x}\big)$ of $\widetilde{\mathrm{Emb}}(M,N)$, the composition
$\widetilde f_{x}=f\circ\widetilde\Phi_{x}^{-1}:
\widetilde\Phi_{x}\big(\widetilde{\mathcal U}_{x}\big)\to\mathfrak Z$
is smooth.

\noindent
If $\mathfrak Z=\R$, then $u=\widetilde\Phi_{x}([y])$ is a critical point
of $\widetilde f_{x}$ if and only if $y$ is a critical point of $f$.
\end{cor}
\begin{proof}
The map $\widetilde f_{x}$ is the restriction to the subspace of $\mathfrak C$-sections of the normal
bundle $x^\perp$ of the smooth function $f_{x}=f\circ\Phi_{x}^{-1}$, thus $\widetilde f_{x}$ is smooth.
For $u\in\widetilde\Phi_{x}\big(\widetilde{\mathcal U}_{x}\big)$, the tangent space at $u$ of
the space of $\mathfrak C$-sections of the bundle $x^*(TN)$ is identified with the space of sections of
some vector subbundle $E$ of $x^*(TN)$ complementary to $\mathrm dx(TM)$ (if $u\ne0$, then $E$
will not necessarily be the normal bundle $x^\perp$).
The invariance property of $f$ says that $\mathrm df_{x}$ vanishes on sections
of the bundle $\mathrm dx(TM)$, from which it follows easily that $u=\widetilde\Phi_{x}([y])$ is a critical point of
$\widetilde f_{x}$ if and only if $y$ is a critical point of $f$.
\end{proof}
\begin{rem}\label{thm:remdiffwidetildeEmbk}
Note that the result of Corollary~\ref{thm:corregularityinvfunct} says in particular that, for
a smooth function $f$ on $\mathrm{Emb}(M,N)$ which is invariant by diffeomorphisms of $M$, one has a well defined
notion of ``critical point of $f$ in $\widetilde{\mathrm{Emb}}(M,N)$''.
We will say that $[y]$ is a critical point of $f$ in $\widetilde{\mathrm{Emb}}(M,N)$ if
given $x:M\to N$ smooth embedding such that $[y]$ belongs to the domain $\widetilde{\mathcal U}_{x}$
of the chart $\widetilde\Phi_{x}$, then $\widetilde\Phi_{x}\big([y]\big)$ is a critical
point of the smooth function $f\circ\widetilde\Phi_{x}^{-1}$. Corollary~\ref{thm:corregularityinvfunct}
says that this notion does not depend on the choice of the chart around $[y]$; of course, this conclusion
could not be drawn using a change of charts argument.
\end{rem}
When $[x]$ is the class of a smooth
embedding $x:M\to N$, then for all questions of differentiability at $[x]$ a it will be convenient
to use the chart $\widetilde\Phi_{x}$, \emph{centered} at the point $x$. The tangent space
at $[x]$ is described in next:
\begin{lem}\label{thm:tangentspace}
Let $x:M\to N$ be a smooth embedding. The tangent space at the
point $[x]$ to $\widetilde{\mathrm{Emb}}(M,N)$
is identified, via the chart $\widetilde\Phi_{x}$ with the Banach space $\mathbf\Gamma(x^\perp)$
of all $\mathfrak C$-sections of the normal bundle
$x^\perp$. If $r\mapsto x_r\in\mathrm{Emb}(M,N)$ is a $\mathcal C^1$-curve with $x_{r_0}=x$ and
$\frac{\mathrm d}{\mathrm dr}\big\vert_{r=r_0}x_r=V\in\mathbf\Gamma\big(x^*(TN)\big)$ then
$r\mapsto\gamma_r=\widetilde\Phi_{x}\big([x_r]\big)$ is of class $\mathcal C^1$, and
$\frac{\mathrm d}{\mathrm dr}\big\vert_{r=r_0}\gamma_r=V^\perp\in\mathbf\Gamma(x^\perp)$,
where $V^\perp(p)$ is the orthogonal projection of $V(p)$ onto the orthogonal space $x^\perp(p)$, $p\in M$.
\end{lem}
\begin{proof}
The domain of $\widetilde\Phi_{x}$ is mapped by $\widetilde\Phi_{x}$ to an open neighborhood of
the zero section of $\mathbf\Gamma(x^\perp)$. The tangent space
is therefore identified with the Banach space itself. Since
$r\mapsto x_r$ is $\mathcal C^1$, then $r\mapsto\eta_r=\Phi_{x}(x_r)$ is $\mathcal C^1$;
now, $\gamma_r=P^\perp(\eta_r)$, where $P^\perp:\mathbf\Gamma\big(x^*(TN)\big)\to\mathbf\Gamma(x^\perp)$
is the bounded linear map defined by $P^\perp(W)=W^\perp$. Thus, $\gamma_r$ is of class $\mathcal C^1$, and
its derivative at $r_0$ is given by $P(V)=V^\perp$.
\end{proof}
\begin{prop}\label{thm:secderinvfunctions}
Let $f:\mathrm{Emb}(M,N)\to\R$ be a smooth function invariant by diffeomorphisms of $M$, and
assume that $x:M\to N$ is a smooth embedding such that $[x]$ is a critical point of
$f$ in $\widetilde{\mathrm{Emb}}(M,N)$ (in the sense of Remark~\ref{thm:remdiffwidetildeEmbk}).
Then, the second variation $\mathrm d^2\big(f\circ\widetilde\Phi_{x}^{-1}\big)(0)$ coincides
with the restriction of the second variation $\mathrm d^2\big(f\circ\Phi_{x}^{-1}\big)(0)$
to the space of $\mathfrak C$-sections of the normal bundle $x^\perp$.
\end{prop}
\begin{proof}
It follows immediately from the fact that, using the local charts $\Phi_{x}$ and
$\widetilde\Phi_{x}$ centered at $x$, then $\widetilde{\mathrm{Emb}}(M,N)$ ($\cong$ sections of the normal
bundle $x^\perp$) is identified with a \emph{linear} subspace of ${\mathrm{Emb}}(M,N)$ ($\cong$ sections of the pull-back bundle $x^*(TN)$).
\end{proof}

\begin{rem}\label{thm:remnotexistencenaturaldiffstruct}
One may wonder whether the set $\widetilde{\mathrm{Emb}}(M,N)$ admits some other \emph{natural} atlas of charts
that are pairwise differentiably compatible, and that make it into a true Banach differentiable manifold.
The existence of such a differentiable structure \emph{fails} if one requires the natural property that
the quotient map $\widetilde\pi:{\mathrm{Emb}}(M,N)\to\widetilde{\mathrm{Emb}}(M,N)$ be a smooth submersion.
Namely, if such a differentiable structure existed, then the inverse image by this projection of points
of $\widetilde{\mathrm{Emb}}(M,N)$, i.e., equivalence classes of embeddings,
would be embedded smooth submanifolds of $\mathrm{Emb}(M,N)$. But as we have observed, equivalence classes of
embeddings $x$ that are only of class $\mathfrak C$ are \emph{not} submanifolds, as they have the same regularity
of the left-composition function $\phi\mapsto x\circ\phi$.
\end{rem}
\end{section}
\begin{section}{Action of the isometry group}
We will now study regularity questions concerning the action of the (connected component of the identity of the)
isometry group $G=\mathrm{Iso}(N,g)$ of the Riemannian manifold
$(N,g)$ on the manifold $\mathrm{Emb}(M,N)$ given by composition on the left, and the corresponding
action on $\widetilde{\mathrm{Emb}}(M,N)$. It is well known (see for instance \cite{Koba}) that
$G$ is a Lie group; if $N$ is compact, then also $G$ is compact.
\begin{prop} The following regularity properties hold for the action of $\mathrm{Iso}(N,g)$.
\begin{enumerate}
\item\label{itm:s1} The action of $\mathrm{Iso}(N,g)$ on $\mathrm{Emb}(M,N)$ is by smooth diffeomorphisms.
\item\label{itm:s2} The corresponding action on $\widetilde{\mathrm{Emb}}(M,N)$ is by homeomorphisms.
\item\label{itm:s3} If $x:M\to N$ is a smooth embedding, then the map
\[\beta_x:\mathrm{Iso}(N,g)\longrightarrow\widetilde{\mathrm{Emb}}(M,N)\]
defined by $\beta_x(\psi)=\psi\cdot[x]$ is a smooth injective immersion on a neighborhood of the identity
(when represented in any of the local charts described in
Subsection~\ref{sub:tilde}).
\item\label{itm:s4} The local charts of $\widetilde{\mathrm{Emb}}(M,N)$ described in Subsection~\ref{sub:tilde},
restricted to the orbit $\mathrm{Iso}(N,g)\cdot[x]$ of a smooth embedding $x:M\to N$ are
differentiably compatible, and they define a differentiable structure on the orbit of $[x]$ in
$\widetilde{\mathrm{Emb}}(M,N)$. The action of $\mathrm{Iso}(N,g)$ on this orbit
is smooth, and this orbit is diffeomorphic to the quotient $\mathrm{Iso}(N,g)/H_x$,
where $H_x$ is the isotropy group of $[x]$.
\end{enumerate}
\end{prop}
\begin{proof}
Isometries of $(N,g)$ are smooth.
Part~\eqref{itm:s1} follows from the fact that left-composition with smooth maps is smooth on
$\mathrm{Emb}(M,N)$; the inverse of left-composition by
$\psi$ is left-composition by $\psi^{-1}$.
As to the map $\widetilde{\mathrm{Emb}}(M,N)\ni [y]\mapsto[\psi\circ y]\in \widetilde{\mathrm{Emb}}(M,N)$, this is continuous, but not smooth.
Namely, given the $\mathfrak C$-section  $u=\widetilde\Phi_{x}(y)$ of $x^\perp$,
then the map $\exp^{-1}\circ\psi\circ\exp(u)$ is a map of class $\mathfrak C$ between open subsets of
$x^\perp$, but it is not a section. Thus, when representing the composition $[\psi\circ y]$ in local charts,
a right composition with a diffeomorphism is needed, which as observed in Subsection~\ref{sub:tilde}
is not smooth, but only continuous. This proves part~\eqref{itm:s2}.
For part~\eqref{itm:s3}, observe that the composition of $\beta_x$ and the local charts $\widetilde\Phi_{y}$
applied to $\psi$ involves only compositions of $\psi$ with smooth diffeomorphism, and it is therefore a smooth
injective immersion of (an open neighborhood of the identity in)
$\mathrm{Iso}(N,g)$. To prove part~\eqref{itm:s4}, observe that, by~\eqref{itm:s3}, the intersection
of the orbit $\mathrm{Iso}(N,g)\cdot[x]$ with the domain of a chart $\widetilde\Phi_{y}$
is an immersed submanifold. Since the orbit of a smooth embedding consists only of classes of smooth embeddings, then
the transition functions restrict to smooth maps at every point of the orbit. Smoothness of the action
on this orbit also follows easily.
\end{proof}
It is an easy observation that, for all $x\in\mathrm{Emb}(M,N)$, the stabilizer
of $[x]$ in $\mathrm{Iso}(N,g)$ is the subgroup consisting of all isometries
$\psi$ that preserve the image $x(M)$, i.e., such that $\psi\big(x(M)\big)=x(M)$.
\end{section}


\begin{thebibliography}{99}

\bibitem{Ebin} \textsc{D. G. Ebin}, \emph{The manifold of Riemannian metrics},  1970
Global Analysis (Proc.\ Sympos.\ Pure Math., Vol.\ XV, Berkeley, Calif., 1968)
pp.\ 11--40 Amer.\ Math.\ Soc., Providence, R.I.

\bibitem{Eliasson} \textsc{H. I. El\u\i asson}, \emph{Geometry of manifolds of maps}, J. Diff.\ Geom.\ \textbf1 (1967), 169--194.

\bibitem{Eells} \textsc{J. Eells}, \emph{On the geometry of function spaces}, Symp.\ Inter.\ de Topologia Alg.\ Mexico, 1957, 303--308.

\bibitem{FreUhl} \textsc{D. S. Freed, K. Uhlenbeck}, \emph{Instantons and four-manifolds}, Second edition.
Mathematical Sciences Research Institute Publications, 1. Springer-Verlag, New York, 1991.
\bibitem{Koba}  \textsc{S. Kobayashi}, \emph{Transformation groups in differential geometry}, Reprint of the 1972 edition.
Classics in Mathematics. Springer-Verlag, Berlin, 1995.

\bibitem{KriMich} \textsc{A. Kriegl, P. Michor}, \emph{The convenient setting for Global Analysis}, M.\ S.\ M.\ vol.\ 53,
Amer.\ Math.\ Soc., Providence, USA, 1997.

\bibitem{Michor} \textsc{P. Michor}, \emph{Manifolds of smooth maps},
Cahiers Topologie G\'eom.\ Diff\'erentielle \textbf{19} (1978), no.\ 1, 47--78.

\bibitem{Michor1} \textsc{P. Michor}, \emph{Manifolds of smooth maps. II. The Lie group
of diffeomorphisms of a noncompact smooth manifold},
Cahiers Topologie G\'eom.\ Diff\'erentielle \textbf{21} (1980), no.\ 1, 63--86.

\bibitem{Michor2} \textsc{P. Michor}, \emph{Manifolds of smooth maps. III.
The principal bundle of embeddings of a noncompact smooth manifold},
Cahiers Topologie G\'eom.\ Diff\'erentielle \textbf{21} (1980), no.\ 3, 325--337.

\bibitem{Palais} \textsc{R. Palais}, \emph{Foundations of Global Nonlinear Analysis}, W. A. Benjamin, 1968.

\bibitem{PalaisTerng} \textsc{R. S. Palais, C´-L. Terng}, \emph{Critical point theory and submanifold geometry}, Lecture Notes in Mathematics, 1353. Springer-Verlag, Berlin, 1988.

\bibitem{PicTau} \textsc{P. Piccione, D. V. Tausk}, \emph{On the Banach differential structure for
sets of maps on non-compact domains}, Nonlinear Anal.\ \textbf{46} (2001), no.\ 2, Ser.\ A: Theory Methods, 245--265.

\bibitem{Whi}
\textsc{B.~White}, \emph{The space of $m$-dimensional surfaces that are stationary for a parametric
elliptic functional}, Indiana Univ.\ Math.\ J.\ \textbf{36} (1987), 567--602.

\bibitem{Whi2}
\textsc{B.~White}, \emph{The space of minimal submanifolds for varying Riemannian
  metrics}, Indiana Univ.\ Math.\ J.\ \textbf{40} (1991), 161--200.
\end{thebibliography}
\end{document}